\documentclass[10pt]{amsart}
\title{On the center of 3-dimensional and 4-dimensional Sklyanin algebras}
\author{Kevin De Laet}
\address{Department of Mathematics, University of Antwerp \\ 
 Middelheimlaan 1, B-2020 Antwerp (Belgium) \\ {\tt kevin.delaet2@uantwerpen.be}}
\date{}
\usepackage{amsmath,amsfonts,array,amscd,amsthm,amssymb,stackrel,verbatim,wrapfig,subfig,float}
\usepackage{enumerate,filecontents}
\usepackage{url}
\usepackage{tikz}
\usepackage[all]{xy}
\usetikzlibrary{arrows,matrix}
\usetikzlibrary{shapes.geometric}
\usetikzlibrary{decorations.markings} 

\tikzset{
  vertice/.style={circle,draw=black},
  decoration={markings,mark=at position 0.5 with {\arrow{>}}}
}

\theoremstyle{plain}

\newcommand{\C}{\mathbb{C}}
\newcommand{\N}{\mathbb{N}}

\newcommand{\Z}{\mathbb{Z}}
\newcommand{\V}{\mathbf{V}}
\newcommand{\VV}{\mathbb{V}_4}

\newcommand{\PP}{\mathbb{P}}

\newtheorem{theorem}{Theorem}
\numberwithin{theorem}{section}
\newtheorem{lemma}[theorem]{Lemma}
\newtheorem{proposition}[theorem]{Proposition}

\newtheorem{remark}[theorem]{Remark}

\theoremstyle{definition}
\newtheorem{mydef}[theorem]{Definition}
\DeclareMathOperator{\Ext}{Ext}

\DeclareMathOperator{\Hom}{Hom}
\DeclareMathOperator{\Aut}{Aut}

\DeclareMathOperator{\Emb}{Emb}

\DeclareMathOperator{\Ann}{Ann}

\numberwithin{equation}{section}
\begin{document}
\begin{abstract}
In this article, a new proof is given of the description of the center of quadratic Sklyanin algebras of global dimension three and four and the center of cubic Sklyanin algebras of global dimension three. The representation theory of the Heisenberg groups $H_2$, $H_3$ and $H_4$ will play an important role. In addition a new proof is given of Van den Bergh's result regarding noncommutative quadrics.
\end{abstract}
\maketitle
\tableofcontents
\section{Introduction}
\label{introduction}
In the first papers about 3-dimensional Sklyanin (\cite{artin1990some} and \cite{artin1991modules}), Artin, Tate and Van den Bergh proved (amongst other things) the following facts:
\begin{itemize}
\item a quadratic Sklyanin algebra $Q_3(E,\tau)$ of global dimension three contains a normal element $c_3 \in (Q_3(E,\tau))_3$ such that $Q_3(E,\tau)/(c_3) \cong {}_3\mathcal{O}_\tau(E,3)$, with ${}_3\mathcal{O}_\tau(E,3)$ the twisted homogeneous coordinate ring of $E$ associated to the automorphism defined by addition with $\tau$ and the line bundle $\mathcal{L}(3O)$, and
\item a cubic Sklyanin algebra $Q_2(E,\tau)$ of global dimension three contains a normal element $c_4 \in (Q_2(E,\tau))_4$ such that $Q_2(E,\tau)/(c_4) \cong {}_2\mathcal{O}_\tau(E)$, with ${}_2\mathcal{O}_\tau(E)$ the twisted homogeneous coordinate ring of $E$ associated to the automorphism defined by addition with $\tau$ and the line bundle $\mathcal{L}(2O)$.
\end{itemize}
In \cite{artin1987graded}, it was shown that $Q_3(E,\tau)$ contains a central element of degree 3 and $Q_2(E,\tau)$ contains a central element of degree 4 by a computer calculation. Consequently, these central elements were equal to $c_3$, respectively to $c_4$ as above, as generically ${}_3\mathcal{O}_\tau(E)$ and ${}_2\mathcal{O}_\tau(E)$ contain no central nontrivial elements.
\par In addition, in \cite{van2011noncommutative} it was proved by Van den Bergh that the second Veronese subalgebra of $Q_2(E,\tau)$, $Q_2(E,\tau)^{(2)}=\oplus_{k\in \N} (Q_2(E,\tau))_{2k}$  is a quotient of a 4-dimensional quadratic Sklyanin algebra associated to the point $2\tau$. This fact in noncommutative projective geometry corresponds to the classical Segre embedding of $\PP^1 \times \PP^1$ in $\PP^3$ from commutative algebraic geometry, or perhaps even better, the embedding of the weighted projective plane $\PP(1,1,2)$ in $\PP^3$. In fact, it was shown in \cite{sklyanin1982some} by a computer calculation that a 4-dimensional Sklyanin algebra has two algebraically independent central elements $\Omega_1,\Omega_2$ of degree two and by results of \cite{smith2013noncommutative}, the quotient map $\pi_{E,\tau}:\xymatrix{ Q_4(E,2\tau) \ar@{->>}[r]& Q_2(E,\tau)^{(2)}}$ has as kernel an ideal generated by an element of $\C \Omega_1 + \C \Omega_2$.
\par The purpose of this paper is to prove the following facts in a conceptual way:
\begin{itemize}
\item the normal element $c_3 \in (Q_3(E,\tau))_3$ is central,
\item the normal element $c_4 \in (Q_2(E,\tau))_4$ is central, and
\item a 4-dimensional Sklyanin algebra $Q_4(E,\tau)$ has 2 algebraically independent elements in the center, $\Omega_1$ and $\Omega_2$ of degree two.
\end{itemize}
The proofs of these 3 theorems depend on the following observations:
\begin{itemize}
\item the Heisenberg group $H_3$ of order 27 acts on $Q_3(E,\tau)$ as gradation preserving automorphisms,
\item the Heisenberg group $H_2$ of order 8, which is isomorphic to $D_4$, the dihedral group of order 8, acts on $Q_2(E,\tau)$ as gradation preserving automorphisms, and
\item the Heisenberg group $H_4$ of order 64 acts on $Q_4(E,\tau)$ as gradation preserving automorphisms.
\end{itemize}
The first and third observation follows from the definition of $n$-dimensional Sklyanin algebras in general as shown in \cite{odeskii1989elliptic}, while the second observation follows directly from the equations. 
In addition, it will be shown that $\C c_3$ is the trivial representation of $H_3$, $\C c_4$ is the trivial representation of $H_2$ and $\C \Omega_1 + \C \Omega_2$ is isomorphic to the unique simple 2-dimensional representation of $H_4$ coming from the quotient map $\xymatrix{H_4 \ar@{->>}[r]& H_2}$ that will be explained below. In fact, for $c_3$ the connection between $c_3$ and the Sklyanin algebra $Q_3(E,2\tau)$ will be explained, finishing the discussion started in \cite{laet2015quotients} and \cite{de2015geometry}.
\par The results needed from previous articles by Artin, Tate and Van den Bergh are
\begin{itemize}
\item the algebras $Q_n(E,\tau)$ for each $n\geq 2$ are noetherian domains (follows from \cite{tate1996homological} for $n\geq 3$ and from \cite{artin1987graded} for $n=2$),
\item there are the following equalities of ideals: $(c_3) =\cap_{P\in E} \Ann_{Q_3(E,\tau)}(M_P)$ and $(c_4) = \cap_{P\in E}\Ann_{Q_2(E,\tau)}(M_P)$ for $M_P$ the point module corresponding to the point $P\in E$ (both follow from \cite{artin1987graded}), and
\item there are the following equalities of Hilbert series
\begin{align*}
H_{Q_2(E,\tau)}(t)&=(1-t)^{-2}(1-t^{-2}), &\text{follows from \cite{artin1987graded}},\\
H_{Q_3(E,\tau)}(t)&=(1-t)^{-3}, &\text{follows from \cite{artin1987graded}},\\
H_{Q_4(E,\tau)}(t)&=(1-t)^{-4}, &\text{follows from \cite{tate1996homological}}.
\end{align*}
\end{itemize}
In addition, a computational proof will be given of Van den Bergh's result that $Q_2(E,\tau)^{(2)}$ is a quotient of $Q_4(E,2\tau)$. The reason for this new proof is to show that the quotient map $\pi_{E,\tau}$ defined above is $\VV$-equivariant with $\VV$ the Klein four-group, with respect to a $\VV$-action that will be defined using the above actions of $H_2$, respectively $H_4$ on the algebras $Q_2(E,\tau)$, respectively $Q_4(E,2\tau)$.
\begin{remark}
Although almost everything works in characteristic not equal to $2$ or $3$ or  for not algebraically closed fields $k$, only the specific case $k = \C$ will be discussed here.
\end{remark}
\subsection{Acknowledgements}
The author would like to thank Paul Smith for various discussions regarding Sklyanin algebras, whose question to find a better proof of the existence of central elements in Sklyanin algebras motivated this paper.
\section{The finite Heisenberg groups}
Let $n \in \N$, $n\geq 2$.
\begin{mydef}
The finite Heisenberg group $H_n$ of order $n^3$ is the group presented by generators and relations
$$
\langle e_1, e_2: e_1^n=e_2^n=[e_1,e_2]^n = 1, [e_1,e_2]\text{ central}\rangle.
$$
\end{mydef}
From the relations, it follows that $H_n$ is a non-trivial central extension of $\Z_n \times \Z_n$ with $\Z_n$, that is,
$$
\xymatrix{1\ar[r]&\Z_n\ar[r]&H_n\ar[r]&\Z_n \times \Z_n \ar[r]&1}.
$$
Let $\omega$ be a primitive $n$th root of unity, then $H_n$ has $n^2$ 1-dimensional representations $\{ {}_n\chi_{i,j}: 0\leq i,j \leq n-1\}$, defined by
$$
{}_n\chi_{i,j}(e_1) = \omega^i, \quad  {}_n\chi_{i,j}(e_2) = \omega^j.
$$
Although the representation theory of $H_n$ can be described in general, only the cases $n=2$, $n=3$ and $n=4$ will be considered here.
\subsection{$H_2$} In this case, $H_2 \cong D_4 = \langle s,t: s^4=t^2=1, tst=s^3 \rangle$ by the isomorphism
$$
\xymatrix{D_4 \ar[r]& H_2},\quad s \mapsto e_1e_2, \quad t \mapsto e_1.
$$
Consequently, $H_2$ has a unique 2-dimensional simple representation ${}_2V = \C x + \C y$, defined by
$$
\begin{array}{cc}
e_1 \cdot x = y, &  e_1 \cdot y = x,\\
e_2 \cdot x = x, &  e_2 \cdot y = -y.
\end{array}
$$
Together with the four 1-dimensional representations, these are all the simple representations of $H_2$.
\subsection{$H_3$}
By the fact that the dimension of a simple representation has to divide the order of the group and that $H_3$ is not commutative, there are two simple representations of dimension 3, ${}_3V_1$ and ${}_3V_2$. Let $\omega$ be a primitive $3$rd root of unity, then ${}_3V_1 = \C x + \C y + \C z$ is defined as the representation
$$
\begin{array}{ccc}
e_1 \cdot x = z, &  e_1 \cdot y = x, & e_1 \cdot z = y,\\
e_2 \cdot x = x, &  e_2 \cdot y = \omega y, & e_2 \cdot z = \omega^2 z.
\end{array}
$$
The representation ${}_3V_2$ is the dual of  ${}_3V_1$, with basis $\{x^*,y^*,z^*\}$ and associated action
$$
\begin{array}{ccc}
e_1 \cdot x^* = z^*, &  e_1 \cdot y^* = x^*, & e_1 \cdot z^* = y^*,\\
e_2 \cdot x^* = x^*, &  e_2 \cdot y^* = \omega^2 y^*, & e_2 \cdot z^* = \omega z^*.
\end{array}
$$
\subsection{$H_4$}
\label{sub:H4rep}
Similar to the case $n=3$, there are two 4-dimensional simple representations 
${}_4V_1= \oplus_{i=0}^3 \C x_i$ and ${}_4V_{3}=\oplus_{i=0}^3 \C x_i^*$ that are dual to each other, with action
$$\begin{array}{cccc}
e_1 \cdot x_0 = x_{3}, & e_1 \cdot x_1 = x_0, & e_1 \cdot x_2 = x_1,& e_1 \cdot x_3 = x_2,\\
e_2 \cdot x_0 = x_0,& e_2 \cdot x_1 = I x_1, & e_2 \cdot x_2 = -x_2,& e_2 \cdot x_3 = -I x_3,
\end{array}
$$
and
$$\begin{array}{cccc}
e_1 \cdot x^*_0 = x^*_{3}, & e_1 \cdot x^*_1 = x^*_0, & e_1 \cdot x^*_2 = x^*_1,& e_1 \cdot x^*_3 = x^*_2,\\
e_2 \cdot x^*_0 = x^*_0,& e_2 \cdot x^*_1 = -I x^*_1, & e_2 \cdot x^*_2 = -x^*_2,& e_2 \cdot x^*_3 = I x^*_3.
\end{array}
$$
In addition, there is a surjective group morphism $\xymatrix{H_4 \ar@{->>}[r]
& H_2}$ with kernel the subgroup $\langle e_1^2,e_2^2,[e_1,e_2]^2 \rangle$. This gives a 2-dimensional simple representation ${}_2V$. Using character theory, one can show that $${}_2V\otimes {}_4\chi_{i,j} \cong{}_2 V\otimes {}_4\chi_{i',j'}\Leftrightarrow (i-i',j-j')\in 2\Z_4 \times 2 \Z_4.$$
This implies that there are four inequivalent 2-dimensional representations. Define ${}_4V_{i,j} = {}_2V \otimes {}_4\chi_{i,j}$ for $(i,j) \in (\Z_4 \times \Z_4)/(2\Z_4 \times 2 \Z_4) \cong \Z_2 \times \Z_2$.
\par By a dimension count, these are all the simple representations of $H_4$.
\section{Artin-Schelter regular algebras}
Recall from \cite{artin1987graded} the definition of an Artin-Schelter regular algebra.
\begin{mydef}
Let $A = \oplus_{n\in \N} A_n$ be a graded algebra with $A_0 = \C$. Then $A$ is an Artin-Schelter regular algebra of global dimension $d$ if
\begin{itemize}
\item $A$ has finite global dimension $d$,
\item $A$ has finite Gelfand-Kirillov dimension, and
\item $A$ is Gorenstein, which means that 
$$
\Ext^q_A(\C,A) =\begin{cases}
0 &\text{ if } q \neq d,\\
\C&\text{ if } q = d,
\end{cases}
$$
with $\C = A/A^+$ as $A$-module.
\end{itemize}
The standard example of an Artin-Schelter regular algebra of global dimension $d$ is the polynomial ring $\C[x_0,\ldots,x_{d-1}]$. A result from \cite{tate1996homological} shows that Sklyanin algebras of any global dimension $d\geq 3$ are Artin-Schelter regular.
\end{mydef}
\par Recall the definition of a point module from \cite{artin1990some}.
\begin{mydef}
Let $A$ be a graded, connected algebra, finitely generated in degree 1. Let $M$ be a graded right $A$-module. Then $M$ is a point module of $A$ if $M$ is cyclic and has Hilbert series $(1-t)^{-1}$.
\end{mydef}
If $A = T(V)/(R)$ with $R$ homogeneous and generated by relations of degree $\geq 2$, then the point modules are determined by the (up to isomorphism) representations of an infinite quiver
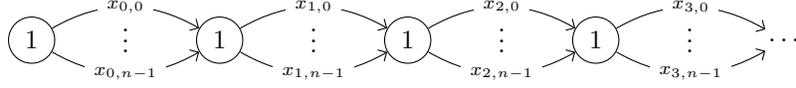
\begin{center}
\begin{tikzpicture}[scale = 2.5]
   \node[vertice] (a) at ( 0, 0) {$1$};
   \node[vertice] (b) at ( 1, 0) {$1$};
   \node[vertice] (c) at ( 2, 0) {$1$};
   \node[vertice] (d) at ( 3, 0) {$1$};
   \node (e) at ( 4, 0) {$\ldots$};
   \node (f) at ( 0.5, 0.04) {$\vdots$};
   \node (g) at ( 1.5, 0.04) {$\vdots$};
   \node (h) at ( 2.5, 0.04) {$\vdots$};
   \node (i) at ( 3.5, 0.04) {$\vdots$};
   \tikzset{every node/.style={fill=white}} 
   \path[->,font=\scriptsize,>=angle 90]
    (a) edge[out=30,in=150]node{$x_{0,0}$} (b)
    (a) edge[out=-30,in=-150]node{$x_{0,n-1}$} (b)
    (b) edge[out=30,in=150]node{$x_{1,0}$} (c)
    (b) edge[out=-30,in=-150]node{$x_{1,n-1}$} (c)
    (c) edge[out=30,in=150]node{$x_{2,0}$} (d)
    (c) edge[out=-30,in=-150]node{$x_{2,n-1}$} (d)
    (d) edge[out=30,in=150]node{$x_{3,0}$} (e)
    (d) edge[out=-30,in=-150]node{$x_{3,n-1}$} (e);
\end{tikzpicture}
\end{center}
with $p_i=[x_{i,0}\ldots:x_{i,n-1}]\in \PP^{n-1}$, with naturally $n = \dim V$. The condition that $M \in \mathbf{GrMod}-A$, with $\mathbf{GrMod}-A$ the category of right graded $A$-modules, then becomes
$$
f \in (R) \text{ of degree }k \Leftrightarrow f^{multi}(p_0,\ldots,p_{k-1})=0.
$$
If $M=\oplus_{n\in \N} M_{n}  \in \mathbf{GrMod}-A$ is a point module, then one defines
$$
M[1] := \oplus_{n\in \N} M_{n+1}
$$
to be the shift of $M$ and by induction, $M[k+1]:=(M[k])[1]$. Under suitable conditions on $A$ (fulfilled by Sklyanin algebras of any global dimension), if $k>0$ is the smallest natural number such that $M[k]\cong M$ in $\mathbf{GrMod}-A$, then $\oplus_{i=0}^{k-1} M[i]$ parametrizes a $\C^*$-orbit of simple representations of dimension $k$, see \cite{smith19924} for a more general treatment.
\begin{proposition}
Let $G \subset \Aut_{gr}(A)$ be a subgroup, with $\Aut_{gr}(A)$ the automorphisms that preserve the gradation. Then the ideal $I = \cap_{M \text{ point module}} \Ann_A(M)$ is a $G$-subrepresentation of $A$, with
$$
\Ann_A(M) = \{x \in A: mx = 0 \text{ for each } m \in M\}.
$$
\label{prop:annpoint}
\end{proposition}
\begin{proof}
Let $M = \oplus_{n\in \N} M_n$ be a point module. For $g \in G$, define a new point module $M^g$ with underlying vector space $M$ but with $A$-action
$m \cdot^g a = m(g\cdot a)$. Then $M^g$ is also a point module of $A$. Consequently, there exists an action of $G$ on the set of point modules of $A$.
\par Let $a \in I$ and let $g \in G$. Let $M$ be a point module and $m \in M$. Then one finds
$$
m(g\cdot a) = m \cdot^g a = 0,
$$
so $g\cdot a \in I$ and consequently, $I$ is a subrepresentation of $A$.
\end{proof}
\begin{proposition}
Let $A$ be a connected graded domain, finitely generated in degree 1 with $\dim A_1 > 1$. Let $G \subset \Aut_{gr}(A)$ be a subgroup such that $A_1$ is a simple representation of $G$. Let $c_k \in A_k$ be a normal element such that $\C c_k$ is a $G$-subrepresentation. Then
\begin{itemize}
\item the automorphism $\sigma \in \Aut_{gr}(A)$ such that $v c_k = c_k \sigma
(v)$ for all $v \in A$ lies in $\C^*$, that is, there exists a $\lambda \in \C^*$ such that $x c_k = \lambda^l c_k x$ for $x \in A_l$, and
\item the element $\lambda$ is a $k$th root of unity.
\end{itemize}
\label{theorem:main}
\end{proposition}
\begin{proof}
If $\C c_k$ is a subrepresentation of $A_k$, then necessarily $k>1$ as $A_1$ was simple and $\dim A_1>1$. For $g \in G$, let $g\cdot c_k = \chi(g)c_k$ for a group homomorphism $\chi:\xymatrix{G \ar[r]& \C^*}$.
Let $g \in G$ and $v \in A_1$, then it follows that 
\begin{align*}
g\cdot (vc_k) &= (g\cdot v)(g\cdot c_k)\\
			  &= \chi(g)c_k \sigma(g\cdot v),\\
g\cdot (vc_k) &= g\cdot (c_k\sigma(v)),\\			
						&= \chi(g)c_k g\cdot\sigma(v).			
\end{align*}
As $A$ is a domain, it follows that for each $v \in A_1$ one has $\sigma(g\cdot v) = g\cdot\sigma(v)$.
\par Consequently, $\sigma \in \Aut_{G}(A_1) = \C^*$ and there exists a $\lambda$ such that $\sigma(x) = \lambda^l x$ for $x \in A_l$. As $\sigma(c_k) = c_k$, it follows that $\lambda$ is a $k$th root of unity.
\end{proof}

\section{Quadratic Sklyanin algebras of global dimension three}
Recall from \cite{artin1987graded} the definition of the quadratic Sklyanin algebras of global dimension three.
\begin{mydef}
Let $[a:b:c] \in \PP^2$, then the quadratic Sklyanin algebra ${}_3S_{a,b,c}$ is the quotient of $\C\langle x,y,z\rangle$ by the quadratic relations
$$
\begin{cases}
ayz+bzy+cx^2,\\
azx+bxz+cy^2,\\
axy+byx+cz^2,
\end{cases}
$$
on the condition that
$$
\V((abc)(X^3+Y^3+Z^3)-(a^3+b^3+c^3)XYZ)=\V(f_{a,b,c}(X,Y,Z))\subset \PP^2_{[X:Y:Z]}
$$
is a smooth curve of genus 1.
\end{mydef}
Let ${}_3R_{a,b,c}$ be the vector space generated by these three equations for $[a:b:c] \in \PP^2$. By taking $O=[1:-1:0]$, $\V(f_{a,b,c}(X,Y,Z))$ obtains the structure of an elliptic curve. Let ${}_3U \subset \PP^2$ be the open subset consisting of points $[a:b:c]$ such that $\V(f_{a,b,c}(X,Y,Z))$ is an elliptic curve.
\par Notice that, if one writes $\C\langle x,y,z\rangle = T({}_3V_1)$ with ${}_3V_1$ the simple $H_3$-module defined above, then the relations of ${}_3S_{a,b,c}$ form an $H_3$-submodule of ${}_3V_1 \otimes {}_3V_1\cong {}_3V_2^{\oplus 3}$. Consequently, $H_3$ acts on ${}_3S_{a,b,c}$ and preserves the gradation.
\par In fact, each submodule of ${}_3V_1 \otimes {}_3V_1$ isomorphic to ${}_3V_1 \wedge {}_3V_1 \cong {}_3V_2$ is determined by a point $[a:b:c] \in \PP^2$ by Schur's lemma, for there are the isomorphisms of varieties
\begin{align*}
\Emb_{H_3}({}_3V_2,{}_3V_2^{\oplus 3}) &= \{ W \subset {}_3V_2^{\oplus 3}: W \cong {}_3V_2 \text{ as } H_3-\text{representation}\}\\
&\cong\Hom^f_{H_3}({}_3V_2,{}_3V_2^{\oplus 3})/\Aut_{H_3}({}_3V_2)\\
&\cong\Hom^f(\C,\C^3)/(\C^*) \cong \PP^2, 
\end{align*}
with $\Hom^f$ and $\Hom^f_G$ being the linear maps, respectively the $G$-linear maps of full rank for $G$ a reductive group. As $\V(f_{a,b,c}(X,Y,Z))$ is generically an elliptic curve, one sees that there is an open subset $U \subset \Emb_{H_3}({}_3V_2,{}_3V_2^{\oplus 3}) \cong \PP^2$ parametrizing Sklyanin algebras.
\par Notice however that there are values of $[a:b:c] \in \PP^2$ such that $\V(f_{a,b,c}(X,Y,Z))$ is not smooth, but ${}_3S_{a,b,c}$ is still an Artin-Schelter regular algebra with Hilbert series $(1-t)^{-3}$. For example, the point $[1:-1:0]$ corresponds to the polynomial ring $\C[x,y,z]$, which is the standard example of an Artin-Schelter regular ring.
\subsection{The point modules} The elliptic curve ${}_3E_{a,b,c}=\V(f_{a,b,c}(X,Y,Z))$ determines the point modules of ${}_3S_{a,b,c}$. In fact, for quadratic Sklyanin algebras of any global dimension not equal to four, it follows from \cite{artin1990some} and \cite{smith1994point} that if $M$ is a point module, then $M$ is completely determined by a point $p_0=p \in E$ and there exists a point $\tau \in E$ such that for each $i\in \N$, $p_i=p+i\tau$. Consequently, $M[i]$ corresponds to the point $p+i\tau$.
\par In the case of global dimension three, things can be made more explicit: from the defining definitions of ${}_3 S_{a,b,c}$ it follows that for a point module $M_p$ the couple $([x_0:y_0:z_0],[x_1:y_1:z_1])$ fulfils the matrix equation
$$
\begin{bmatrix}
cx_0&bz_0&ay_0\\az_0&cy_0&bx_0\\by_0&ax_0&cz_0
\end{bmatrix}
\begin{bmatrix}
x_1\\y_1\\z_1
\end{bmatrix}=0.
$$
As such, $[x_1:y_1:z_1]$ can be calculated by taking the cross product of 2 of the 3 rows of the first matrix. Calculating this for the point $O = [1:-1:0]$, one sees that $\tau = [a:b:c]$.
\subsection{The center} Let ${}_3I_{a,b,c} = \cap_{p\in {}_3E_{a,b,c}}\Ann_{{}_3S_{a,b,c}}(M_p)$, then it was proved in \cite{artin1990some} that ${}_3I_{a,b,c}$ is generated by a normal element  ${}_3c_{a,b,c}$ of degree 3 such that ${}_3 S_{a,b,c}/(_{3}c_{a,b,c}) \cong {}_3\mathcal{O}_\tau({}_3E_{a,b,c})$, the twisted homogeneous coordinate ring associated to $E$ and the automorphism defined by addition with $\tau$.
\begin{lemma}
The vector space $\C({}_3c_{a,b,c})$ is an $H_3$-subrepresentation of ${}_3S_{a,b,c}$.
\label{lemma:subrep3}
\end{lemma}
\begin{proof}
The conditions of Proposition \ref{prop:annpoint} are fulfilled, as the set
$$
\{M_p: p \in {}_3E_{a,b,c}\}
$$
is the entire set of point modules of ${}_3S_{a,b,c}$.
\end{proof}
In fact, one can explicitly determine the $H_3$-representation $\C({}_3c_{a,b,c})$.
\begin{proposition}
The vector space $\C{}_3c_{a,b,c}$ is isomorphic to the trivial representation of $H_3$. In addition, taking the basis
$$
\begin{cases}
f_1 = zxy+xyz+yzx,\\
f_2 = yxz+zyx+xzy,\\
f_3 = x^3+y^3+z^3
\end{cases}
$$
of $({}_3V_1^{\otimes 3})^{H_3}$, then ${}_3c_{a,b,c}$ is the image of the element
$$
\alpha f_1+ \beta f_2 + \gamma f_3
$$
with $[\alpha:\beta:\gamma]= -2\tau$ by the addition law on ${}_3E_{a,b,c}$.
\label{prop:descriptionnormal}
\end{proposition}
\begin{proof}
Consider the variety
$$
\mathcal{E}_{a,b,c} = \{(P,Q,R) \in {}_3E_{a,b,c}^3:Q-P=R-Q=\tau\} \subset \PP^2_{[x_0:y_0:z_0]}\times \PP^2_{[x_1:y_1:z_1]}\times \PP^2_{[x_2:y_2:z_2]},
$$
which is isomorphic to ${}_3E_{a,b,c}$. Take the multi-linearizations $\{f_1^{multi},f_2^{multi},f_3^{multi}\}$. Then $$
A f_1+ B f_2 + C f_3 \in (R_{a,b,c},{}_3c_{a,b,c}))  \Leftrightarrow A f^{multi}_1+ B f^{multi}_2 + C f^{multi}_3 \in \mathbf{I}(\mathcal{E}_{a,b,c}),
$$
with $\mathbf{I}(\mathcal{E}_{a,b,c})$ the ideal of $\C[x_0,y_0,z_0,x_1,y_1,z_1,x_2,y_2,z_2]$ of all polynomials vanishing on $\mathcal{E}_{a,b,c}$.
\par Consider the projection to the first and the third factor of $\mathcal{E}_{a,b,c}$ 
$$\pi_{0,2}(\mathcal{E}_{a,b,c})=\{(P,R) \in {}_3E_{a,b,c}^3:R-P=2\tau\} \subset \PP^2_{[x_0:y_0:z_0]} \times \PP^2_{[x_2:y_2:z_2]}.$$
Consequently, if $[\alpha:\beta:\gamma]= -2\tau$, then the following 3 functions belong to $\mathbf{I}(\mathcal{E}_{a,b,c})$:
$$
\begin{cases}
\beta y_0 z_2 + \alpha z_0 y_2 + \gamma x_0x_2,\\
\beta z_0 x_2 + \alpha x_0 z_2 + \gamma y_0y_2,\\
\beta x_0 y_2 + \alpha y_0 x_2 + \gamma z_0z_2.
\end{cases}
$$
Consequently, the following element belongs to $\mathbf{I}(\mathcal{E}_{a,b,c})$:
$$
(\beta y_0 z_2 + \alpha z_0 y_2 + \gamma x_0x_2)x_1+(\beta z_0 x_2 + \alpha x_0 z_2 + \gamma y_0y_2)y_1 + (\beta x_0 y_2 + \alpha y_0 x_2 + \gamma z_0z_2)z_1,
$$
which leads to the claimed element $\alpha f_1 + \beta f_2 +\gamma f_3$. One still has to check that $\alpha f_1 + \beta f_2 +\gamma f_3 \not\in (R_{a,b,c})$.
\par By the results of \cite{artin1990some}, $R_{a,b,c}\otimes {}_3V_1 + {}_3V_1\otimes R_{a,b,c}$ is 17-dimensional and 
$$R_{a,b,c}\otimes {}_3V_1 \cap {}_3V_1\otimes R_{a,b,c}=\C (a f_1 + b f_2 + c f_3).$$
As ${}_3V_1 \otimes {}_3V_2 \cong \oplus_{i,j=0}^2 {}_3\chi_{i,j}$, it follows that 
$$
\dim (R_{a,b,c}\otimes {}_3V_1 + {}_3V_1\otimes R_{a,b,c})^{H_3}=1.	
$$
Consequently, $\alpha f_1 + \beta f_2 +\gamma f_3 \not\in (R_{a,b,c})$ if $[a:b:c] \neq [\alpha:\beta:\gamma]$, which is equivalent to saying $3\tau \neq O$. But if $\tau$ is of order 1 or 3, then ${}_3S_{a,b,c}$ is not a Sklyanin algebra.
\end{proof}

\begin{theorem}
The normal element ${}_3c_{a,b,c}$ is a central element of ${}_3S_{a,b,c}$.
\label{th:3dimcentral}
\end{theorem}
\begin{proof}
The conditions of Theorem \ref{theorem:main} are fulfilled: 
\begin{itemize}
\item the algebra ${}_3S_{a,b,c}$ is a graded domain, finitely generated in degree 1 with $H_3 \subset \Aut_{gr}({}_3S_{a,b,c})$ such that $({}_3S_{a,b,c})_1 = {}_3V_1$ is a simple representation of $H_3$, and
\item the element ${}_3c_{a,b,c}$ is an $H_3$-subrepresentation of ${}_3S_{a,b,c}$.
\end{itemize}
Consequently, one can define the function
$$
\lambda: \xymatrix{{}_3 U \ar[r]& \Z_3}, [a:b:c] \mapsto \lambda_{a,b,c},
$$
with $\lambda_{a,b,c}$ the third root of unity such that $v{}_3c_{a,b,c} = \lambda_{a,b,c}{}_3c_{a,b,c}v$ for all $v \in {}_3 V_1$.
\par This function is a morphism as there exists a presentation of ${}_3c_{a,b,c}$ in $({}_3V_1)^{\otimes 3}$ that is continuous and never becomes 0 in ${}_3S_{a,b,c}$. As ${}_3U$ is irreducible, this function is constant. If one takes now a Sklyanin algebra associated to a point of order two, one gets a Clifford algebra of rank 3 over the polynomial ring $\C[x^2,y^2,z^2]$. Consequently, the point modules of such a Clifford algebra are annihilated by a central element of degree 3 (see \cite{lebruyn1995central}). This implies that on $\V(a-b)\cap U$, $\lambda([a:a:c])=1$, so $\lambda$ is the constant 1-function.
\end{proof}
\section{Cubic Sklyanin algebras of global dimension three}
Again from \cite{artin1987graded}, recall the definition of the cubic Sklyanin algebras of global dimension three.
\begin{mydef}
Let $[a:b:c] \in \PP^2$, then the cubic Sklyanin algebra ${}_2S_{a,b,c}$ is the quotient of $\C \langle x,y \rangle$ by the relations
$$
\begin{cases}
a(y^2x+xy^2)+byxy+cx^3,\\
a(x^2y+yx^2)+bxyx+cy^3,
\end{cases}
$$
on the condition that 
$$
{}_2E_{a,b,c}=\V((b^2-c^2)x_0y_0x_1y_1+ab(x_0^2y_1^2+y_0^2x_1^2)-ac(x_0^2x_1^2+y_0^2y_1^2))\subset \PP^1_{[x_0:y_0]}\times \PP^1_{[x_1:y_1]}
$$
is a smooth curve of genus 1.
\end{mydef}
Let ${}_2R_{a,b,c}$ be the vector space generated by these relations. As in the quadratic case, the Heisenberg group $H_2$ of order 8 acts on ${}_2S_{a,b,c}$ as gradation preserving automorphisms such that $({}_2S_{a,b,c})_1={}_2 V$ and $R_{a,b,c} \cong {}_2V$. However, while in the quadratic case the generic embedding of ${}_3R_{a,b,c}$ in ${}_3V_1 \otimes {}_3V_1$ corresponds to a Sklyanin algebra, this is no longer true in the cubic case. In fact, one has ${}_2V^{\otimes 3} \cong {}_2V^{\oplus 4}$, so that 
$$
\Emb_{H_2}({}_2V,{}_2V^{\oplus 4}) \cong \PP^3,
$$
but the generic element of this variety corresponds to an algebra with a degree 4 vector space of dimension 8 instead of 9, so the generic element can not correspond to an Artin-Schelter regular algebra by \cite{artin1987graded}.
\par As $[e_1,e_2]$ acts by multiplication with $-1$ on ${}_2 V$, the second Veronese subalgebra ${}_2S_{a,b,c}^{(2)}$ is an $H_2/\langle [e_1,e_2]\rangle\cong \Z_2 \times \Z_2 = \VV$-module.
As such, it is easy to see that ${}_2 V \otimes {}_2V$ decomposes as
$$
{}_2V \otimes {}_2V = \C(x^2+y^2) \oplus \C(x^2-y^2)\oplus \C(xy+yx)\oplus \C(xy-yx).
$$
Let $w_{00} = x^2+y^2$, $w_{10}=x^2-y^2$, $w_{01}=xy+yx$ and $w_{11}=xy-yx$, then  one finds
$$
w_{ij}\cong {}_2\chi_{i,j} \text{ for } (i,j) \in \Z_2 \times \Z_2.
$$
\par Define 
\begin{align*}
\gamma(x)&=a(y^2x+xy^2)+byxy+cx^3,\\
\gamma(y)&=a(x^2y+yx^2)+bxyx+cy^3,
\end{align*}
then the linear extension of $\gamma$ from $x$, $y$ to a linear function from ${}_2V$ to ${}_2V^{\otimes 3}$ is $H_2$-linear.
\par Writing $\gamma(x)$ and $\gamma(y)$ as elements of ${}_2V^{\otimes 2} \otimes {}_2V$ and ${}_2V \otimes {}_2V^{\otimes 2}$, one finds
\begin{align*}
2\gamma(x)&=(a+c)w_{00}x+(c-a)w_{10}x+(a+b)w_{01}y+(a-b)w_{11}y\\&=(a+c)xw_{00}+(c-a)xw_{10}+(a+b)yw_{01}+(b-a)yw_{11},\\
2\gamma(y)&=(a+c)w_{00}y+(a-c)w_{10}y+(a+b)w_{01}x+(b-a)w_{11}x\\&=(a+c)yw_{00}+(a-c)yw_{10}+(a+b)xw_{01}+(a-b)xw_{11}.
\end{align*}

\subsection{The point modules}
As in the quadratic case, the elliptic curve ${}_2E_{a,b,c}$ associated to ${}_2S_{a,b,c}$ is defined by taking the multi-linearization of the relations of ${}_2S_{a,b,c}$ and writing these as a matrix equation
$$
\begin{bmatrix}
ay_0y_1+cx_0x_1 & ax_0y_1 + b y_0x_1\\
ay_0x_1+bx_0y_1 & ax_0x_1+cy_1y_1
\end{bmatrix}
\begin{bmatrix}
x_2 \\ y_2
\end{bmatrix}=0.
$$
Consequently, the point variety of ${}_2S_{a,b,c}$ is determined by the determinant of this $2 \times 2$-matrix as a subvariety of $\PP^1_{[x_0:y_0]}\times \PP^1_{[x_1:y_1]}$, which one can check is an elliptic curve.
\par However, ${}_2S_{a,b,c}$ can be an Artin-Schelter regular algebra although it is not necessarily a Sklyanin algebra, for example for the point $[1:-2:0]$ one gets the relations of $U(\mathfrak{h})$, with $\mathfrak{h}$ the 3-dimensional Heisenberg Lie algebra $\C x + \C y + \C z$ with Lie bracket
$$
[x,y]=z,\quad [x,z]=0, \quad [y,z]=0.
$$
\subsection{The center}
By \cite{artin1990some}, there exists a degree 4 normal element ${}_2c_{a,b,c}$, such that $\cap_{p\in {}_2E_{a,b,c}}\Ann_{{}_2 S_{a,b,c}}(M_p) = ({}_2c_{a,b,c})$.
\begin{lemma}
The vector space $\C({}_2c_{a,b,c})$ is an $H_2$-subrepresentation of ${}_2S_{a,b,c}$.
\end{lemma}
\begin{proof}
Again, the conditions of Proposition \ref{prop:annpoint} are fulfilled.
\end{proof}
In fact, as in the 3-dimensional case, $\C({}_2c_{a,b,c})$ is the trivial representation of $H_2$. The proof of this fact will follow from the $\VV$-equivariant algebra map that will be constructed in Theorem \ref{th:cubic4dim}
$$
\pi_{a,b,c}:\xymatrix{{}_4S_{\alpha_1,\alpha_2,\alpha_3}\ar[r]&{}_2S_{a,b,c}^{(2)}}.
$$
Using Theorem \ref{th:cubic4dim}, which will be proved later, one can show that
a presentation of ${}_2c_{a,b,c}$ as an element of ${}_2V^{\otimes 4}$ is given by
\begin{align*}
{}_2c_{a,b,c}=&b(a^2-c^2)((xy)^2+(yx)^2)+a(b^2-a^2)(yx^2y+xy^2x)+\\&a(c^2-a^2)(x^2y^2+y^2x^2)+c(a^2-b^2)(x^4+y^4),
\end{align*}
which is (up to a scalar) the same element found in \cite{artin1987graded}. Unfortunately, at the moment, no intrinsic description of the coefficients of this element is known as in the quadratic case discussed in Proposition \ref{prop:descriptionnormal}.
\par Notice that ${}_2c_{a,b,c} \rightarrow [x,y]^2$ if $[a:b:c]\rightarrow [1:-2:0]$. Let ${}_2U'\subset \PP^2$ be the dense subset parametrizing Artin-Schelter regular algebras, then $[1:-2:0] \in {}_2U'$.
\begin{theorem}
The element ${}_2c_{a,b,c}$ is central in ${}_2 S_{a,b,c}$ if ${}_2 S_{a,b,c}$ is  Artin-Schelter regular.
\end{theorem}
\begin{proof}
If ${}_2c_{a,b,c}=0$ then there is nothing to prove, so assume ${}_2c_{a,b,c}\neq 0$ for an open subset ${}_2U'' \subset U'$. By Proposition \ref{theorem:main}, for each $[a:b:c] \in {}_2U''$ there exists a fourth root of unity $\lambda_{a,b,c}$ such that for each $v \in {}_2V$, $v{}_2c_{a,b,c} = \lambda_{a,b,c}{}_2c_{a,b,c}v$.
As in the quadratic case, let 
$$
\lambda: \xymatrix{{}_4 U'' \ar[r]& \Z_4}, [a:b:c] \mapsto \lambda_{a,b,c}.
$$
be the corresponding function, then $\lambda$ is a constant function. As $\lambda([1:-2:0])=1$, it follows that $\lambda$ is the constant 1-function.
\end{proof}
\section{Quadratic Sklyanin algebras of global dimension four}
Recall the definition of the four dimensional Sklyanin algebra $Q_4(E,\tau)$ from \cite{odeskii1989elliptic}.
\begin{mydef}
Let $E = \C/(\Z + \Z\eta)$ be an elliptic curve with $\Im(\eta) >0$ and let $\tau \in E$. For each $\alpha \in \Z_4$, let $\theta_\alpha(z)$ be a theta function of order $4$ with respect to the lattice $\Z + \Z \eta$, that is, $\{\theta_0(z),\theta_1(z),\theta_2(z),\theta_{3}(z)\}$ are (non-constant) holomorphic functions $\xymatrix{\C \ar[r]& \C}$ defined by the following functional equations
\begin{align*}
\theta_\alpha\left(z+\frac{1}{4}\right) &= \exp\left(2 \pi I\frac{\alpha}{4}\right)\theta_\alpha(z),\\
\quad\theta_\alpha\left(z+\frac{\eta}{4}\right)&=\exp\left(-2\pi I z-\frac{\pi I}{4}+\frac{3\pi I\eta}{4}\right)\theta_{\alpha+1}(z),
\end{align*}
for each $\alpha \in \Z_4$ and indices taken in $\Z_4$. Then the Sklyanin algebra $Q_4(E,\tau)$ is defined as the quotient of $\C\langle x_1,x_1,x_2,x_3 \rangle$ by the quadratic relations
$$
\oplus_{(i,j) \in \Z_4\times \Z_4\setminus \Delta}\C\left(\sum_{r \in \Z_4} \frac{1}{\theta_{j-i-r}(-\tau)\theta_{r}(\tau)} x_{j-r}x_{i+r}\right).
$$
\end{mydef}
From this definition, it follows that
\begin{itemize}
\item the group $H_4$ acts on $Q_4(E,\tau)$ such that the degree 1 part is isomorphic to ${}_4V_1$, and
\item there is the $H_4$-isomorphism ${}_4V_1 \wedge {}_4V_1 \cong {}_4R_\tau(E)$ with 
${}_4R_\tau(E) \subset {}_4V_1 \otimes {}_4V_1$ the space of quadratic relations of $Q_4(E,\tau)$.
\end{itemize}
It will now be discussed how to find the relations of $Q_4(E,\tau)$ in general using the representation theory of $H_4$ and these two facts, without resorting to $\theta$-functions. As one wants to find $\VV$-equivariant maps, it is useful to find a basis for ${}_4V_1$ of eigenvectors with respect to the $\VV \cong \langle e_1^2,e_2^2\rangle \subset H_4$-action.
\par Take for ${}_4V_1$ the canonical basis $\{x_0,x_1,x_2,x_3\}$ as in Subsection \ref{sub:H4rep}. Taking the basis $v_{00} = x_0 + x_2,v_{10} = x_0 - x_2,v_{01} = x_1 + x_3, v_{11} = x_1 - x_3$, then the action of $H_4$ is determined by
$$
\begin{array}{cccc}
e_1 \cdot v_{00} = v_{01},& e_1 \cdot v_{10} = -v_{11},&e_1 \cdot v_{01} = v_{00}, &e_1 \cdot v_{11} = v_{10},\\
e_2 \cdot v_{00} = v_{10},& e_2 \cdot v_{10} = v_{00},&e_2 \cdot v_{01} = Iv_{11}, &e_2 \cdot v_{11} = Iv_{01}.
\end{array}
$$
Then it is easy to see that
$$
e_1^2 \cdot v_{ij} = (-1)^{i}v_{ij}, \quad e_2^2 \cdot v_{ij} = (-1)^{j}v_{ij} \text{ for } (i,j) \in \Z_2 \times \Z_2.
$$
\par A calculation shows that
\begin{align*}
{}_4V_1 \otimes {}_4V_1 &\cong {}_4V_{0,0}^{\oplus 2} \oplus {}_4V_{1,0}^{\oplus 2} \oplus {}_4V_{0,1}^{\oplus 2} \oplus {}_4V_{1,1}^{\oplus 2},\\
{}_4V_1 \wedge {}_4V_1 &\cong {}_4V_{1,0} \oplus {}_4V_{0,1} \oplus {}_4V_{1,1}.
\end{align*}
From these decompositions, it follows that 
$$
\Emb_{H_4}({}_4V_1 \wedge {}_4V_1, {}_4V_1 \otimes {}_4V_1) \cong \PP^1 \times \PP^1 \times \PP^1.
$$
A (generic) subspace of ${}_4V_1 \otimes {}_4V_1$ $H_4$-isomorphic to ${}_4V_1 \wedge {}_4V_1$ is determined by the following six relations (for $(\lambda_{10},\lambda_{01},\lambda_{11}) \in \C^3$)
\begin{align*}
\begin{cases}
[v_{00},v_{10}]-\lambda_{10} \{v_{01},v_{11}\}\\
[v_{01},v_{11}]+\lambda_{10}\{v_{00},v_{10}\}
\end{cases},\\
\begin{cases}
[v_{00},v_{01}]-\lambda_{01} \{v_{11},v_{10}\}\\
[v_{11},v_{10}]+\lambda_{01}\{v_{00},v_{01}\}
\end{cases},\\
\begin{cases}
[v_{00},v_{11}]-\lambda_{11} \{v_{10},v_{01}\}\\
[v_{10},v_{01}]-\lambda_{11}\{v_{00},v_{11}\}
\end{cases}.
\end{align*}
Let ${}_4S_{\lambda_1,\lambda_2,\lambda_3}$ be the corresponding quotient of $\C\langle v_{00},v_{10},v_{01},v_{11}\rangle$.
\par However, not every triple $(\lambda_{10},\lambda_{01},\lambda_{11}) \in \C^3$ determines a Sklyanin algebra. Only the triples $(\lambda_{10},\lambda_{01},\lambda_{11})$ such that the $4 \times 4$-minors of the matrix
$$
\begin{bmatrix}
-v_{10} & v_{00}&-\lambda_{10}v_{11}&-\lambda_{10}v_{01}\\
\lambda_{10}v_{10} & \lambda_{10}v_{00}&-v_{11}&v_{01}\\
-v_{01} & -\lambda_{01}v_{11} & v_{00} & -\lambda_{01} v_{10}\\
\lambda_{01}v_{01}&v_{11}&\lambda_{01}v_{00}&-v_{10}\\
-v_{11}&-\lambda_{11}v_{01}& -\lambda_{11}v_{10}&v_{00}\\
-\lambda_{11}v_{11}&-v_{01}& v_{10}&-\lambda_{11}v_{00}
\end{bmatrix}
$$
vanish on an elliptic curve $E_{\lambda} \subset \PP^3$ with $\lambda \in \C$, which will be embedded as the intersection of the 2 quadrics
$$
\begin{cases}
v_{00}^2+v_{10}^2-\lambda(v_{01}^2-v_{11}^2)=0,\\
v_{01}^2+v_{11}^2-\lambda(v_{00}^2-v_{10}^2)=0.
\end{cases}
$$
\begin{proposition}
The condition that the triple $(\lambda_{10},\lambda_{01},\lambda_{11})$ generically determines a Sklyanin algebra is given by
$$
\lambda_{10}^2 + \lambda_{01}^2-\lambda_{11}^2-(\lambda_{10}\lambda_{01}\lambda_{11})^2=0,
$$
with corresponding elliptic curve $E_\lambda$ with 
$$
\lambda = \frac{\lambda_{10}(\lambda_{01}\lambda_{11}+1)}{\lambda_{01}-\lambda_{11}}
$$
in its point variety.
\end{proposition}
\begin{proof}
A computer calculation.
\end{proof}

\par In order to find all presentations of four dimensional Sklyanin algebras with the degree 1 part decomposed in four 1-dimensional representations of $\VV$, one still needs to allow for base change with the subgroup $\Aut_{\VV}({}_4V_1)$. By Schur's lemma, $\Aut_{\VV}({}_4V_1)\cong (\C^*)^4$, but the subgroup of scalar matrices $\C^*$ acts trivially as a Sklyanin algebra is graded.
Let $v_{ij} = t_{ij} v'_{ij}$ for $(i,j) \in \Z_2 \times \Z_2 \setminus\{(0,0)\}$ and $v_{00}=v'_{00}$. Then the relations become
\begin{align*}
\begin{cases}
[v'_{00},v'_{10}]-\lambda_{10}\frac{t_{01}t_{11}}{t_{10}}\{v'_{01},v'_{11}\}\\
[v'_{01},v'_{11}]+\lambda_{10}\frac{t_{10}}{t_{01}t_{11}}\{v'_{00},v'_{10}\}
\end{cases},\\
\begin{cases}
[v'_{00},v'_{01}]-\lambda_{01}\frac{t_{11}t_{10}}{t_{01}}\{v'_{11},v'_{10}\}\\
[v'_{11},v'_{10}]+\lambda_{01}\frac{t_{01}}{t_{11}t_{10}}\{v'_{00},v'_{01}\}
\end{cases},\\
\begin{cases}
[v'_{00},v'_{11}]-\lambda_{11}\frac{t_{10}t_{01}}{t_{11}}\{v'_{10},v'_{01}\}\\
[v'_{10},v'_{01}]-\lambda_{11}\frac{t_{11}}{t_{10}t_{01}}\{v'_{00},v'_{11}\}
\end{cases}.
\end{align*}
Taking the variables
$$
\begin{array}{ll}
a_{10}=\lambda_{10}\frac{t_{01}t_{11}}{t_{10}}, & b_{10}=-\lambda_{10}\frac{t_{10}}{t_{01}t_{11}},\\
a_{01}=\lambda_{01}\frac{t_{11}t_{10}}{t_{01}}, & b_{01}=-\lambda_{01}\frac{t_{01}}{t_{11}t_{10}},\\
a_{11}=\lambda_{11}\frac{t_{10}t_{01}}{t_{11}}, & b_{11}=\lambda_{11}\frac{t_{11}}{t_{10}t_{01}},
\end{array}
$$
one finds the following proposition.
\begin{proposition}
For generic values of $(a_{10},a_{01},a_{11},b_{10},b_{01},b_{11}) \in \C^6$ lying on the fivefold
$$
\V_5=\V(a_{10}b_{10}+a_{01}b_{01}+a_{11}b_{11}+a_{10}b_{10}a_{01}b_{01}a_{11}b_{11}),
$$
the algebra with generators $v_{00},v_{10},v_{01},v_{11}$ and relations
\begin{align}
\begin{cases}
[v_{00},v_{10}]-a_{10}\{v_{01},v_{11}\}\\
[v_{01},v_{11}]-b_{10}\{v_{00},v_{10}\}
\end{cases},\\
\begin{cases}
[v_{00},v_{01}]-a_{01}\{v_{11},v_{10}\}\\
[v_{11},v_{10}]-b_{01}\{v_{00},v_{01}\}
\end{cases},\\
\begin{cases}
[v_{00},v_{11}]-a_{11}\{v_{10},v_{01}\}\\
[v_{10},v_{01}]-b_{11}\{v_{00},v_{11}\}
\end{cases},
\end{align}
is a four dimensional Sklyanin algebra and each four dimensional Sklyanin algebra arises this way.
\label{prop:4dimsklyanin}
\end{proposition}
Usually, the intersection of $\V_5$ with $\V(b_{10}-1,b_{01}-1,b_{11}-1)$ is taken as variety parametrizing four dimensional Sklyanin algebras, as in for example \cite{smith1992regularity}. In this case, ${}_4 S_{\alpha_1,\alpha_2,\alpha_3}$ will denote the algebra with parameters
$$
\begin{array}{ll}
a_{10}=\alpha_1, & b_{10}=1,\\
a_{01}=\alpha_2, & b_{01}=1,\\
a_{11}=\alpha_3, & b_{11}=1.
\end{array}
$$
The condition that ${}_4 S_{\alpha_1,\alpha_2,\alpha_3}$ is (generically) a Sklyanin algebra then becomes $\alpha_1+\alpha_2+\alpha_3+\alpha_1\alpha_2\alpha_3=0$.
\par For any sextuple $(a_{10},a_{01},a_{11},b_{10},b_{01},b_{11}) \in \C^6$, the corresponding algebras that have $H_4$ in its automorphism group can be recognized by the rule
$$
a_{ij} = 0 \Leftrightarrow b_{ij}=0 \text{ for } (i,j) \in \Z_2 \times \Z_2 \setminus\{(0,0)\}.
$$
\subsection{The point modules}
Following \cite{smith1992regularity}, it follows that the point modules of a Sklyanin algebra ${}_4 S_{\lambda_{10},\lambda_{01},\lambda_{11}}$ are not only determined by $E_\lambda$, but there are also four point modules, corresponding to the $\C^*$-orbits of simple representations of $({}_4 S_{\lambda_{10},\lambda_{01},\lambda_{11}})_{ab}={}_4 S_{\lambda_{10},\lambda_{01},\lambda_{11}}/[{}_4 S_{\lambda_{10},\lambda_{01},\lambda_{11}},{}_4 S_{\lambda_{10},\lambda_{01},\lambda_{11}}]$. From the relations, it is easy to see that
$$
({}_4 S_{\lambda_{10},\lambda_{01},\lambda_{11}})_{ab}= \C[v_{00},v_{10},v_{01},v_{11}]/(v_{ij}v_{kl}, (i,j)\neq(k,l)).
$$
As such, ${}_4 S_{\lambda_{10},\lambda_{01},\lambda_{11}}$ has four $\C^*$-families of 1-dimensional representations.
\subsection{The center}
The tricks that were used in the global dimension three case don't work here, as the normal element in this case won't be a representation of $H_4$. Luckily, the fact that the abelianization of ${}_4 S_{\lambda_{10},\lambda_{01},\lambda_{11}}$ is not finite dimensional will help.
\par First, one has to construct a normal element of degree two. One can do this using Theorem \ref{th:cubic4dim}.
\begin{lemma}
The Sklyanin algebra ${}_4 S_{\lambda_{10},\lambda_{01},\lambda_{11}}$ has a normal element of degree two.
\end{lemma}
\begin{proof}
This follows directly from Hilbert series considerations and the fact that ${}_4 S_{\lambda_{10},\lambda_{01},\lambda_{11}}$ is a domain, by taking the kernel of the quotient map $$\pi_{a,b,c}:\xymatrix{ {}_4 S_{\alpha_1,\alpha_2\alpha_3}\ar@{->>}[r]& {}_2S_{a,b,c}^{(2)}}.$$
\end{proof}
\begin{theorem}
There are two central elements $\Omega_1,\Omega_2$ of degree two in a four dimensional Sklyanin algebra ${}_4 S_{\lambda_{10},\lambda_{01},\lambda_{11}}$. Generically these two elements generate the center and they are algebraically independent. In addition, 
$$
(\Omega_1,\Omega_2) = \cap_{p \in E_\lambda}\Ann_{{}_4 S_{\lambda_{10},\lambda_{01},\lambda_{11}}} M_p
$$
with as before $M_p$ the point module corresponding to $p \in E_\lambda$.
\end{theorem}
\begin{proof}
Let $\Omega_1$ be the normal element from the previous lemma. As ${}_4 S_{\lambda_{10},\lambda_{01},\lambda_{11}}$ is a domain, there exists an automorphism $\sigma \in \Aut_{gr}({}_4 S_{\lambda_{10},\lambda_{01},\lambda_{11}})$ such that $v\Omega_1 = \Omega_1 \sigma(v)$ for each $v \in {}_4 S_{\lambda_{10},\lambda_{01},\lambda_{11}}$.
\par Looking at the ring $({}_4 S_{\lambda_{10},\lambda_{01},\lambda_{11}})_{ab}$ one sees that for each $v \in {}_4V_1$
$$
\overline{\Omega_1}\overline{v}=\overline{v}\overline{\Omega_1} = \overline{\Omega_1}\overline{\sigma(v)}.
$$
The element $\Omega_1$ does not annihilate any of the simple representations coming from  $({}_4 S_{\lambda_{10},\lambda_{01},\lambda_{11}})_{ab}$, for if it did then such a $\C^*$-family of simple representations would correspond to a point module of ${}_2 S_{a,b,c}^{(2)}$ for some $[a:b:c]\in \PP^2$. But the point modules of ${}_2 S_{a,b,c}^{(2)}$  are the same as those of ${}_2 S_{a,b,c}$ which were already shown to be parametrized by ${}_2E_{a,b,c}\cong E_\lambda$ for some $\lambda \in \C$.
\par Consequently, $\Omega_1$ is not a zero divisor in $({}_4 S_{\lambda_{10},\lambda_{01},\lambda_{11}})_{ab}$, so that $\overline{v} = \overline{\sigma(v)}$ for each $v \in {}_4V_1$, which implies that $\sigma$ is the identity and $\Omega_1$ is therefore central.
\par As each $H_4$-subrepresentation of ${}_4V_1 \otimes {}_4V_1$ is 2-dimensional, it follows that there exists at least one other element in the center, let $\Omega_2$ be such a central element, linearly independent of $\Omega_1$. Then both $\Omega_1$ and $\Omega_2$ annihilate all the point modules lying on $E_\lambda$, as $E_\lambda$ is an $H_4$-set. Due to the fact that ${}_2S_{a,b,c}$ is a domain, it follows that $\Omega_1$ and $\Omega_2$ are algebraically independent.
\par From this, it follows that
$$
(\Omega_1,\Omega_2) = \cap_{p \in E_\lambda}\Ann_{{}_4 S_{\lambda_{10},\lambda_{01},\lambda_{11}}} M_p
$$
due to the fact that the quotient ${}_4 S_{\lambda_{10},\lambda_{01},\lambda_{11}}/(\Omega_1,\Omega_2)$ has Hilbert series
$$
1+\sum_{k=1}^\infty 4kt^k=\frac{(1+t)^2}{(1-t)^2}.
$$
\end{proof}
\section{Noncommutative quadrics lie on noncommutative $\PP^3$'s}
\begin{theorem}
The second Veronese subalgebra ${}_2 S_{a,b,c}^{(2)}$ of a cubic Sklyanin algebra ${}_2 S_{a,b,c}$ is a quotient of the four dimensional Sklyanin algebra ${}_4S_{\alpha_1,\alpha_2,\alpha_3}$ with
$$
\alpha_1 = \frac{bc}{a^2}, \quad \alpha_2 = -\frac{(b+c)^2-4a^2}{(b-c)^2}, \quad \alpha_3 = \frac{(b-c)^2-4a^2}{(b+c)^2}.
$$
In addition, the quotient map $\pi_{a,b,c}:\xymatrix{{}_4S_{\alpha_1,\alpha_2,\alpha_3} \ar@{->>}[r]& {}_2 S_{a,b,c}^{(2)}}$ is $\VV$-equivariant with respect to the induced action of $\VV$ on ${}_4S_{\alpha_1,\alpha_2,\alpha_3}$ coming from the inclusion $\VV \subset H_4$ and the action of $\VV$ on ${}_2S_{a,b,c}^{(2)}$ coming from the isomorphism $H_2/([e_1,e_2])\cong \VV$.
\label{th:cubic4dim}
\end{theorem}
\begin{proof}
As before, let ${}_2R_{a,b,c}$ be the 2-dimensional subspace of ${}_2V^{\otimes 3}$ generated by the relations of ${}_2 S_{a,b,c}$. Then ${}_2R_{a,b,c} \otimes {}_2V + {}_2V \otimes {}_2R_{a,b,c}$ decomposes as $\VV$-module as
$$
{}_2R_{a,b,c} \otimes {}_2V + {}_2V \otimes {}_2R_{a,b,c} \cong {}_2\chi_{0,0}\oplus{}_2\chi_{1,0}^{\oplus 2}\oplus{}_2\chi_{0,1}^{\oplus 2}\oplus{}_2\chi_{1,1}^{\oplus 2}.
$$
So to construct a quotient map $\pi:\xymatrix{{}_4S_{\alpha_1,\alpha_2,\alpha_3} \ar@{->>}[r]& {}_2S_{a,b,c}}$ for some value of $(\alpha_1,\alpha_2,\alpha_3)$ that is $\VV$-equivariant, one needs to decompose ${}_2R_{a,b,c} \otimes {}_2V + {}_2V \otimes {}_2R_{a,b,c}$ as $\VV$-representation with respect to the variables $w_{00},w_{10},w_{01},w_{11}$ from above. By a calculation, one finds for the $\VV$-submodule isomorphic to ${}_2\chi_{1,0}^{\oplus 2}$:
\begin{align*}
2(\gamma(x)x-\gamma(y)y) &= (a+c)w_{00}w_{10}+(c-a)w_{10}w_{00}-(a+b)w_{01}w_{11}+(a-b)w_{11}w_{01},\\
2(x\gamma(x)-y\gamma(y)) &= (a+c)w_{10}w_{00}+(c-a)w_{00}w_{10}+(a+b)w_{11}w_{01}+(b-a)w_{01}w_{11}.
\end{align*}
These relations generate the same vector space as
$$
\begin{cases}
a[w_{00},w_{10}]-b\{w_{01},w_{11}\},\\
a[w_{01},w_{11}]-c\{w_{00},w_{10}\}.
\end{cases}
$$
This leads to the claimed value $\alpha_1 = \frac{bc}{a^2}$.
\par Similar, for the $\VV$-submodule isomorphic to ${}_2\chi_{0,1}^{\oplus 2}$:
\begin{align*}
2(\gamma(x)y+\gamma(y)x) &= (a+c)w_{00}w_{01}+(c-a)w_{10}w_{11}+(a+b)w_{01}w_{00}+(b-a)w_{11}w_{10},\\
2(y\gamma(x)+x\gamma(y)) &= (a+c)w_{01}w_{00}+(a-c)w_{11}w_{10}+(a+b)w_{00}w_{01}+(a-b)w_{10}w_{11},
\end{align*}
which leads to the vector space generated by
$$
\begin{cases}
(c-b)[w_{00},w_{01}]+(b+c-2a)\{w_{11},w_{10}\},\\
(b-c)[w_{11},w_{01}]+(b+c+2a)\{w_{00},w_{01}\},
\end{cases}
$$
which gives the value $\alpha_2 = -\frac{(b+c)^2-4a^2}{(b-c)^2}$.
\par Lastly, for the submodule isomorphic to ${}_2\chi_{1,1}^{\oplus 2}$:
\begin{align*}
2(\gamma(x)y-\gamma(y)x) &= (a+c)w_{00}w_{11}+(c-a)w_{10}w_{01}-(a+b)w_{01}w_{10}+(a-b)w_{11}w_{00},\\
2(x\gamma(y)-y\gamma(x)) &= (a+c)w_{11}w_{00}+(a-c)w_{01}w_{10}+(a+b)w_{10}w_{01}+(a-b)w_{00}w_{11},
\end{align*}
which leads to the relations
$$
\begin{cases}
(b+c)[w_{00},w_{11}]+(-2a-b+c)\{w_{10},w_{01}\},\\
(b+c)[w_{10},w_{01}]+(2a-b+c)\{w_{00},w_{11}\},
\end{cases}
$$
which leads to the value $\alpha_3 = \frac{(b-c)^2-4a^2}{(b+c)^2}$.
\par A calculation shows that for these values, $\alpha_1+\alpha_2+\alpha_3 + \alpha_1 \alpha_2 \alpha_3 = 0$, so $ {}_2S_{a,b,c}$ is indeed a quotient of the 4-dimensional Sklyanin algebra ${}_4S_{\alpha_1,\alpha_2,\alpha_3}$.
\end{proof}
Let ${}_4 S_{a,b,c}$ be the four dimensional Sklyanin algebra with the values of $a_{ij}$ and $b_{ij}$ from the proof of the previous theorem. Then the extra relation one needs to go from ${}_4 S_{a,b,c}$ to ${}_2 S_{a,b,c}$ is given by
\begin{equation}
\mathfrak{n}_{a,b,c}=(a+c)v_{00}^2+(c-a)v_{10}^2+(a+b)v_{01}^2+(b-a)v_{11}^2.
\label{extrarelation}
\end{equation}
As this element is fixed under the action of $\VV \subset H_4$, it follows that $H_4 \cdot \mathfrak{n}_{a,b,c} \cong {}_4V_{0,0}$ as $H_4$-representation.
\par Consequently, as $\pi_{a,b,c}$ is $\VV$-equivariant, if one takes $\Omega_1 = \mathfrak{n}_{a,b,c}$ and $\Omega_2$ any linearly independent element in $H_4 \cdot \Omega_1$, then $\pi_{a,b,c}(\Omega_2) = \mu  {}_2 c_{a,b,c}$ for some $\mu \in \C^*$ and $\C ({}_2c_{a,b,c})$ is isomorphic to ${}_2 \chi_{0,0}$.
\bibliographystyle{abbrv}
\bibliography{computationalproof}

\end{document}